\documentclass[10pt]{amsart}
\usepackage{eurosym}
\usepackage{amsfonts}
\usepackage{amssymb}
\usepackage{graphicx}
\usepackage{amsmath,amsfonts,amsmath,amssymb}
\usepackage{mathrsfs}
\usepackage[latin1]{inputenc}
\usepackage[T1]{fontenc}
\usepackage{color,xspace}
\usepackage{enumerate}
\usepackage[usenames,dvipsnames]{pstricks}
\newtheorem{theorem}{Theorem}[section]

\newtheorem{corollary}[theorem]{Corollary}

\newtheorem{lem}[theorem]{Lemma}

\newtheorem{proposition}[theorem]{Proposition}
\newtheorem{remark}[theorem]{Remark}

\newtheorem{Assumption}[theorem]{Assumption}

\newcommand{\R}{\mathbb{R}}
\newcommand{\C}{\mathbb{C}}

\newcommand{\N}{\mathbb{N}}

\newcommand{\norm}[1]{\left\Vert#1\right\Vert}

\newcommand{\abs}[1]{\left\vert#1\right\vert}

\usepackage{times}

\usepackage{epsfig}

\begin{document}
\title{The inverse problem of two-state quantum systems with non-adiabatic static linear coupling}
\author{Andrii Khrabustovskyi, Imen Rassas and \'Eric Soccorsi}

\begin{abstract}
We consider the inverse problem of determining the coupling coefficients in a
two-state Schr\"odinger system. We prove a Lipschitz stability inequality for
the zeroth and first order coupling terms by finitely many partial lateral
measurements of the solution to the coupled Schr\"odinger equations.\\


\end{abstract}

\maketitle

\section{Introduction}
Let $\Omega$ be a bounded domain of $\mathbb{R}^n$, $n \in \mathbb{N}=\{1,2,\ldots \}$, with smooth boundary $\Gamma=\partial\Omega$. Given $T\in (0,+\infty)$, we consider the following initial-boundary value problem (IBVP) for the coupled two-state Schr\"odinger equations in the unknowns $u^\pm=u^\pm(x,t)$,
\begin{equation}
\label{a1}
 \left\{
\begin{array}{ll}
-i\partial_t u^+ -\Delta u^+ + q^+u^+ + A\cdot\nabla u^- + pu^- =0 &\textrm{in}\ Q=\Omega \times (0,T), \\
-i\partial_t u^- -\Delta u^- + q^-u^- -A\cdot\nabla u^+ + pu^+ = 0 &\textrm{in}\ Q, \\
u^+(\cdot,0)=u_0^+,\ u^-(\cdot,0)=u_0^- &\textrm{in}\ \Omega,\\
u^+=g^+,\ u^-=g^-&\textrm{on}\ \Sigma=\Gamma \times (0,T),
\end{array}
\right.
\end{equation}
where $u_0^\pm$ and $g^\pm$ are suitable initial states and Dirichlet boundary conditions, respectively. Here, $A  : \Omega \to \R^n$, $p : \Omega \to \R$ and $q^\pm : \Omega \to \R$ are all real-valued.  In this article we are concerned by the stability issue in the inverse problem of determining the unknown functions $A$, $p$ and $q^\pm$ from a finite number of local boundary measurements of the solution to \eqref{a1}.\\

The IBVP \eqref{a1} describes the dynamics of a two-state (or two-level) quantum system. This terminology is justified by the fact that the quantum system modeled by \eqref{a1} can exist in any superposition of the two independent (in the sense that they can be physically distinguished) states $u^\pm$. 
As a matter of fact particles such as electrons, neutrinos or protons, are fermions and they have a two-state quantum mechanical label called spin. In Quantum Mechanics, the spin is an intrinsic form of angular momentum carried by elementary particles and the spin of fermions is half-integer. Namely, the electron is a spin-$1/2$ particle, i.e. the spin of the electron can have values $\hslash / 2$ (spin up) or $-\hslash /2$ (spin down), where $\hslash$ is the reduced Planck constant. Notice about this that for the sake of simplicity, the various physical constants appearing in \eqref{a1}, such as $\hslash$, the mass of the particle or its charge, are all taken equal to $1$ in this text.
In \eqref{a1} the dynamics of the two states $u^\pm$ are bound together through non-adiabatic linear coupling $p u^\mp  \pm A \cdot \nabla u^\mp$, see \cite{LSY} and the references therein for the relevance of non-adiabatic processes in physics or reactive chemistry. Gradient coupling 
appears also naturally in quantum fields theory (see \cite{BLMM,R}) or quantum cosmology (see \cite{DKT, BTW}), and it can sometimes be seen as a first-order approximation of nonlinear coupling (see \cite{Y}).

\subsection{What is known so far: A short bibliography}

There is a wide mathematical literature on inverse coefficient problems for the dynamic Schr\"odinger equation. Without tying to be exhaustive, one may mention \cite{BC, B, CKS, BKS2, KS}. In all these papers, an infinite number of boundary observations of the solution is required, but in \cite{BP,YY}, the real-valued electric potential of the Schr\"odinger equation is Lipschitz stably retrieved from a single partial boundary measurement. This result was improved in \cite{MOR} to smaller partial measurements and extended in \cite{HKSY} to complex-valued electric potentials. The method used in \cite{BP,YY, MOR, HKSY} is based essentially on an appropriate Carleman estimate. We refer to \cite{HKSY,T,YY} for actual examples of this inequality for the Schr\"odinger equation.
The idea of using a Carleman estimate for solving inverse problems first appeared in A. L. Bukhgeim and M. V. Klibanov paper \cite{BK}. Since its inception in 1981, this technique has then been widely and successfully applied by numerous authors to parabolic or hyperbolic systems, to the dynamic Schr\"odinger equation, and even to coupled systems of PDEs. See \cite{K} and references therein, for a complete review of multidimensional inverse problems solved by the Bukhgeim-Klibanov method.

Notice that in \cite{BP,YY,MOR,HKSY}, the data are measured on a part of the boundary that fulfills a geometric condition related to geometric optics condition insuring observability. This condition was relaxed in \cite{BC} for a  real-valued electric potential, under the assumption that the potential is known in the vicinity of the boundary. We refer to \cite{KPS1,KPS2,BKS1} for the same type of inverse problems but stated in an infinite cylindrical domain.
The problem of stably determining the space varying part (resp., static) magnetic potential of the autonomous (resp., non-autonomous) Schr\"odinger equation is treated in \cite{CS} (resp., \cite{HKSY}). In both cases, the $n$-th dimensional unknown magnetic vector potential, $n \geq 1$, is recovered from $n$ partial Neumann data, obtained by $n$-times suitably changing the initial condition attached at the magnetic Schr\"odinger equation.
All the above mentioned papers are concerned with the "one state" Schr\"odinger equation.  In \cite{LT} the authors show unique determination of the static electric coupling potential in a two state magnetic Schr\"odinger equation, by one partial measurement of the solution. Otherwise stated, assuming that the gradient coupling potential is known, \cite{LT} claims that knowledge of one partial Neumann data uniquely determines the scalar coupling potential. In the present paper, the framework is the same as in \cite{LT} but with uniformly zero magnetic field, and we investigate the stability issue in the inverse problem of identifying both the electric and the gradient coupling potentials, by finitely many partial boundary observations of the solution.

\subsection{Notations}

Throughout the entire text $x = (x_1,...,x_n)$ is a generic point of $\Omega$ and we set $\partial_i = \frac{\partial}{\partial x_i}$
and $\partial_{i,j}^2=\frac{\partial^2}{\partial x_i \partial x_j}$ for all $i,j=1,\ldots,n$.  We write $\partial_{i}^2$ instead of $\frac{\partial^2}{\partial x_i^2}$ and as usual, $\Delta$ denotes the Laplace operator, i.e. $\Delta=\partial_1^2+\ldots+\partial_n^2$. Next, for any multi-index $k = (k_1,\ldots,k_n) \in \N_0^n$, where $\N_0 = \{ 0 \} \cup \N$, we put $\abs{k}=k_1+\ldots+k_n$ and $\partial_x^k= \partial_1^{k_1} \ldots \partial_n^{k_n}$. Similarly, we write $\partial_t = \frac{\partial}{\partial t}$.

Further, the symbol $\cdot$ denotes the scalar product in $\C^m$, $m \geq 1$, and $\abs{\zeta}= \sqrt{\zeta \cdot \zeta}$ for all $\zeta \in \C^m$. For any row vector $a=(a_1,a_2,\ldots,a_m)$ we write $a^T$ for the transpose of $a$ in such a way that
$\nabla=(\partial_1,\ldots,\partial_n)^T$ is the gradient operator with respect to $x$. Further, $\nabla \cdot$ denotes the divergence operator and we set $\partial_\nu u = \nabla u \cdot \nu$, where $\nu$ is the outward normal vector to the boundary $\Gamma$.

Let us now introduce the following functional spaces. For any manifold $X$, we set 
$$ H^{r,s}(X \times (0,T)) = L^2(0,T;H^r(X)) \cap H^s(0,T;L^2(X)),\ r, s >0, $$
where $H^s(X)$ denotes the usual Sobolev space on $X$ of order $s$. For the sake of notational simplicity, we write $H^{r,s}(Q)$ (resp., $H^{r,s}(\Sigma)$) instead of $H^{r,s}(\Omega \times (0,T))$ (resp., $H^{r,s}(\Gamma \times (0,T))$). The space $H^{r,s}(X \times (0,T))$ is endowed with the norm $\norm{\cdot}_{H^{r,s}(X \times (0,T))}=\norm{\cdot}_{L^2(0,T;H^r(X))}+\norm{\cdot}_{H^s(0,T;L^2(X))}$ and we recall from \cite[Section 4, Theorem 2.1]{LM2} that for all $u \in H^{r,s}(X \times (0,T))$, $r,\ s>0$, and all $(j,k) \in \N_0^n \times\N_0$ such that $1-\frac{\abs{j}}{r}-\frac{k}{s}>0$, we have
\begin{equation}
\label{n1} 
\partial_x^j \partial_t^k u\in H^{\mu,\nu}(X \times (0,T))\ \mbox{with}\ \frac{\mu}{r}=\frac{\nu}{s}=1- \frac{\abs{j}}{s} - \frac{k}{s}, 
\end{equation}
and the estimate
\begin{equation}
\label{n2} 
\norm{\partial_x^j \partial_t^k u}_{H^{\mu,\nu}(X \times (0,T))} \leq \norm{u}_{H^{r,s}(X \times (0,T))}.
\end{equation}

\subsection{Main results}
Prior to examining the inverse problem under consideration in this article, we treat the well-posedness issue for the IBVP \eqref{a1}. Let $N$ be the unique natural number satisfying
$$ N \in \N \cap \left( \frac{n+2}{4}+1,\frac{n+2}{4}+ 2 \right]. $$
Then we have the following existence, uniqueness and regularity result for the solution to IBVP \eqref{a1}.

\begin{proposition}
\label{pr1}
Assume that $\Gamma$ is $C^{2(N+1)}$ and pick $M \in (0,+\infty)$. Let $A \in W^{2N+1,\infty}(\Omega,\mathbb{R}^n)$ verify $\nabla \cdot A=0$ a.e. in $\Omega$ and let
$p\in W^{2N+1,\infty}(\Omega,\mathbb{R})$ and $q^\pm \in W^{2N+1,\infty}(\Omega,\mathbb{R})$ be such a way that
$$ \norm{A}_{W^{2N+1,\infty}(\Omega)^n}+ \norm{p}_{W^{2N+1,\infty}(\Omega)}+\norm{q^+}_{ W^{2N+1,\infty}(\Omega)}+ \norm{q^-}_{W^{2N+1,\infty}(\Omega)} \leq M.$$
Then, for all
$g=(g^+,g^-)^T \in H^{2(N+7/4),N+7/4}(\Sigma)^2$ and all $u_0=(u_0^+,u_0^-)^T \in H^{2N+3}(\Omega)^2$, fulfilling the following compatibility conditions 
\begin{equation}
\label{d4}
\partial_t^\ell g(\cdot,0)=( -i )^\ell \left( \begin{array}{cc} -\Delta  + q^+ & A \cdot \nabla + p \\ -A \cdot \nabla +p & -\Delta+ q^- \end{array} \right)^\ell u_0\ \mbox{on}\ \Gamma,\ \ell=0,\cdots,N,
\end{equation}
the IBVP \eqref{a1} admits a unique solution $u=(u^+,u^-)^T \in W^{1,\infty}(0,T;W^{1,\infty}(\Omega)^2)$. Moreover,  there exists a positive constant $C$, depending only on $\Omega$, $T$, $M$, $u_0$ and $g$, such that 
\begin{equation}
\label{d4c}
\norm{u}_{W^{1,\infty}(0,T;W^{1,\infty}(\Omega)^2)} \leq C.
\end{equation}
\end{proposition}

Armed with Proposition \ref{pr1}, we may now state the main result of this article.
As a preamble we introduce the sets of admissible unknown coefficients $A$, $p$, $q^\pm$. To this purpose, for $M \in(0,\infty)$, $A_0\in W^{2N+1,\infty}(\Omega,\R^n)$ and $q_0 \in W^{2N+1,\infty}(\Omega,\R)$, we define
\begin{enumerate}[i)]
\item the set of admissible unknown gradient vector potentials as
$$\mathcal{A}_M(A_0)=\{A\in W^{2N+1,\infty}(\Omega,\mathbb{R}^n),\ \norm{A}_{W^{2N+1,\infty}(\Omega)^n}\leq M,\ \nabla\cdot A=0\ \textrm{and}\ \partial_x^k A= \partial_x^k A_0\ \textrm{on}\ \Gamma,\ \abs{k} \leq 2(N-1) \},$$
\item and the set of admissible unknown electric potentials as
$$\mathcal{Q}_M(q_0)=\{ q\in W^{2N+1,\infty}(\Omega,\mathbb{R}),\ \norm{q}_{W^{2N+1,\infty}(\Omega)}\leq M\ \textrm{and}\ \partial_x^k q= \partial_x^k q_0\ \textrm{on}\ \Gamma,\ \abs{k} \leq 2(N-1) \}.$$
\end{enumerate}
Next, $p_0$ and $q_0^\pm$ being fixed in $W^{2N+1,\infty}(\Omega,\R)$, the main result of this article is as follows.

\begin{theorem}
\label{th1}
Assume that $\Gamma$ is $C^{2(N+1)}$ and for $j=1,2$, let $A_j\in\mathcal{A}_M(A_0)$, $p_j\in\mathcal{Q}_M(p_0)$ and $q_j^\pm \in\mathcal{Q}_M(q^\pm_0)$. Then, there exist $n+2$ initial states $u_0^k=(u_0^{+,k},u_0^{-,k})^T \in H^{2N+3}(\Omega)^2$ and boundary conditions $g^k=(g^{+,k},g^{-,k})^T \in H^{2(N+7 \slash 4), N + 7 \slash 4}(\Sigma)^2$, $k=1,\ldots, n+2$, fufilling the compatibility conditions 
\begin{equation}
\label{cck}
\partial_t^\ell g^k(\cdot,0)=( -i )^\ell \left( \begin{array}{cc} -\Delta  + q_0^+ & A_0 \cdot \nabla + p_0 \\ -A_0 \cdot \nabla +p_0 & -\Delta+ q_0^- \end{array} \right)^\ell u_0^k\ \mbox{on}\ \Gamma,\ \ell=0,\cdots,N,
\end{equation}
such that we have
\begin{eqnarray*}
& & \norm{A_1-A_2}^2_{L^2(\Omega)^n}+\norm{p_1-p_2}^2_{L^2(\Omega)}+\norm{q_1^+-q_2^+}^2_{L^2(\Omega)}+\norm{q_1^- - q_2^-}^2_{L^2(\Omega)}\\
& \leq & C \sum_{k=1}^{n+2} \left( \norm{\partial_\nu \partial_t u_1^{-,k}- \partial_\nu \partial_t u_2^{-,k}}^2_{L^2(\Sigma_*)}
+\norm{\partial_\nu \partial_t u_1^{+,k} -\partial_\nu \partial_t u_2^{+,k}}^2_{L^2(\Sigma_*)} \right),
\end{eqnarray*}
for some positive constant $C$, depending only on $\Omega$, $T$, $M$ and $(u_0^{\pm,k},g^{\pm,k})$ for $k=1,\ldots,n+2$. Here, $u_j^k=(u_j^{+,k},u_j^{-,k})^T$, for $j=1,2$ and $k=1,\ldots,n+2$, is the solution to \eqref{a1} given by Proposition \ref{pr1}, where \hfill \break 
$(A_j,p_j,q_j^\pm,u_0^{\pm,k},g^{\pm,k})$ is substituted for $(A,p,q^\pm,u_0^{\pm},g^{\pm})$. 
\end{theorem}

Theorem \ref{th1} claims Lipschitz stable recovery of $n+3$ unknown functions ($p$ and $q^\pm$ and the $n$ components of $A$) by $n+2$ local boundary measurements of the solution $u=(u^+,u^-)$ to \eqref{a1}. Bearing in mind that $A$ is divergence free and that its trace on $\Gamma$ is prescribed, this amounts to saying that $n+2$ unknown scalar functions can be stably retrieved by the same number of local Neumann data. From this viewpoint, the result of Theorem \ref{th1} is thus optimal.

\subsection{Outline}

The derivation of Proposition \ref{pr1} can be found in Section \ref{sec-dp} while Section \ref{sec-ip} contains the proof of Theorem \ref{th1}. Finally, several technical results used for establishing that the elliptic part of the Schr\"odiger equation \eqref{a1} is self-adjoint in $L^2(\Omega)^2$, are collected in the Appendix.

\section{Analysis of the direct problem}
\label{sec-dp}
In this section we prove Proposition \ref{pr1}.

\subsection{Preliminaries: Self-adjointness and basic regularity}

\subsubsection{Self-adjointness} 
In this section we assume that $\Gamma$ is $C^2$, that $A \in L^\infty(\Omega,\R^n)$ is gradient free, i.e. that $\nabla \cdot A =0$ a.e. in $\Omega$, and that $p \in L^\infty(\Omega,\R)$ and $q^\pm \in L^\infty(\Omega,\R)$. 

Let $\Delta^D$ denote the Dirichlet-Laplacian in $L^2(\Omega)$, with domain $\mathcal{D}_0=H_0^1(\Omega) \cap H^2(\Omega)$. Since $\Gamma$ is $C^2$ then it is well known that $\Delta^D$ is self-adjoint in $L^2(\Omega)$. As a consequence, the operator
$$ \Delta^D u = (\Delta^D u^+, \Delta^D u^-)^T,\ u=(u^+,u^-)^T \in \mathcal{D}_0^2, $$
is self-adjoint in $L^2(\Omega)^2$.
Put
$$
\tilde{A} =\begin{pmatrix}
0&A \\
-A&0\\
\end{pmatrix},\ 
\tilde{p}=\begin{pmatrix}
0&p\\
p&0\\
\end{pmatrix}\ \mbox{and}\
\tilde{q}=\begin{pmatrix}
q^+&0 \\
0&q^-\\
\end{pmatrix}.
$$
Since $p\in L^\infty(\Omega,\R)$ (resp., $q^\pm L^\infty(\Omega,\R)$) then the multiplication operator by $\tilde{p}$ (resp., $\tilde{q}$), defined by $\tilde{p} u = (p u^-, p u^+)^T$ (resp.,  $\tilde{q} u = (q^- u^+, q^- u^-)^T$) for all $u=(u^+,u^-)^T \in L^2(\Omega)^2$, and denoted by $\tilde{p}$ (resp., $\tilde{q}$) is symmetric in $L^2(\Omega)^2$. Similarly, since $\nabla \cdot A=0$, then we infer from the Stokes formula that the operator 
$$ \tilde{A} \cdot \nabla u =  (A \cdot \nabla u^-,-A \cdot \nabla u^+)^T,\ u=(u^+,u^-)^T \in H_0^1(\Omega)^2, $$
is symmetric in $L^2(\Omega)^2$ as well (see Appendix A in Section \ref{sec-aA}). As a consequence, the operator $\tilde{A} \cdot \nabla +\tilde{p} + \tilde{q}$, with domain $H_0^1(\Omega)^2$, is symmetric in $L^2(\Omega)^2$. Moreover, we know from Appendix B (see Section \ref{sec-aB}) that $\tilde{A} \cdot \nabla + \tilde{p} + \tilde{q}$ is $\Delta^D$-bounded in $L^2(\Omega)^2$, with relative bound zero: For any $\epsilon \in (0,1)$, there exists $C_\epsilon>0$, depending only on $\epsilon$, $\norm{A}_{L^{\infty}(\Omega)}$, $\norm{p}_{L^{\infty}(\Omega)}$ and $\norm{q^\pm}_{L^{\infty}(\Omega)}$, such that 
$$ \norm{(\tilde{A} \cdot \nabla +\tilde{p} + \tilde{q}) u}_{L^2(\Omega)^2} \leq \epsilon \norm{\Delta^D u}_{L^2(\Omega)^2} + C_\epsilon \norm{u}_{L^2(\Omega)^2},\ u \in \mathcal{D}_0^2. $$
Therefore, the Kato-Rellich Theorem (see \cite[Theorem X.12]{RS2}) yields the following:

\begin{lem}
\label{lm2}
Assume that  $\Gamma$ is $C^2$, that $A \in L^\infty(\Omega,\R^n)$ fulfills $\nabla \cdot A =0$ a.e. in $\Omega$, and that $p \in L^\infty(\Omega,\R)$ and $q^\pm \in L^\infty(\Omega,\R)$. Then the operator 
$H(A,q^\pm,p)= -\Delta^D  + \tilde{q} + \tilde{A} \cdot \nabla + \tilde{p}$, with domain $D(H(A,q^\pm,p))=\mathcal{D}_0^2$, is self-adjoint in $L^2(\Omega)^2$.
\end{lem}

Otherwise stated, $H(A,q^\pm,p)$ is the self-adjoint realization in $L^2(\Omega)^2$ of the formal operator acting in $(C_0^\infty(\Omega)^2)'$, $\mathcal{H}(A,q^\pm,p)= - \Delta + \tilde{q} + \tilde{A} \cdot \nabla + \tilde{p}$, endowed with homogeneous boundary conditions on $\Gamma$. We point out for further use that $u=(u^+,u^-)^T$ solves \eqref{a1} may be equivalently rewritten as $u$ is solution to the IBVP
\begin{equation}
\label{a1b}
\left\{ \begin{array}{lc} -i \partial_t u + \mathcal{H}(A,q^\pm,p) u  =0 & \mbox{in}\ Q, \\ u = g & \mbox{on}\ \Sigma, \\ u(\cdot,0)=u_0 & \mbox{in}\ \Omega. \end{array}\right.
\end{equation}

\subsubsection{Existence, uniqueness and basic regularity result}
In this section we establish the following existence and uniqueness result by adapting the analysis carried out in \cite[Section 2]{HKSY} to the coupled system \eqref{a1}. 


\begin{lem}
\label{lm3}
Assume that $\Gamma$, $A$, $p$ and $q^\pm$ are the same as in Lemma \ref{lm2}. Then for all $g=(g^+,g^-)^T \in H^{7/2,7/4}(\Sigma)^2$ and all $u_0=(u_0^+,u_0^-)^T \in H^3(\Omega)^2$, the IBVP \eqref{a1} admits a unique solution $u=(u^+,u^-)^T \in H^{2,1}(Q)^2$ to \eqref{a1}. Moreover, there exists a constant $C$, depending only on $\Omega$, $T$ and $M$, such that 
\begin{equation}
\label{d0}
\norm{u} _{H^{2,1}(Q)^2} \leq C \left( \norm{g}_{H^{7/2,7/4}(\Sigma)^2} + \norm{u_0}_{H^3(\Omega)^2} \right). 
\end{equation}
\end{lem}
\begin{proof}
Since $g=(g^+,g^-)^T \in H^{7/2,7/4}(\Sigma)^2$ and $u_0=(u_0^+,u_0^-)^T \in H^3(\Omega)^2$ then, in virtue of \cite[Section 4, Theorem 2.1]{LM2}, there exists $G =(G^+,G^-)^T \in H^{4,2}(Q)^2$ such that $G=g$ on $\Sigma$ and $G(\cdot,0)=u_0$ in $\Omega$. Moreover, we have \begin{equation}
\label{d1}
\norm{G}_{H^{4,2}(Q)^2} \leq C \left( \norm{u_0}_{H^3(\Omega)^2} + \norm{g}_{H^{7/2,7/4}(\Sigma)^2} \right),
\end{equation}
for some constant $C>0$, depending only on $\Omega$, $T$ and $M$.

Evidently, $u=(u^+,u^-)^T$ is solution to \eqref{a1} if and only if $v=u-G=(u^+-G^+,u^- - G^-)^T$ is solution to the following Cauchy problem
\begin{equation}
\label{d2}
 \left\{
\begin{array}{ll}
-i \partial_t v + H(A,q^\pm,p) v =f &\textrm{in}\ Q, \\
v(\cdot,0)=0 &\textrm{in}\ \Omega,
\end{array}
\right.
\end{equation}
where $f=(f^+,f^-)^T=-(-i \partial_t - \Delta + \tilde{q}) G$. Further, with reference to \cite[Proposition 2.3]{LM2} we have $\partial_t G \in H^{2,1}(Q)^2$ with $\norm{\partial_t G}_{H^{2,1}(Q)^2} \leq C \norm{G}_{H^{4,2}(Q)^2}$, whence $f \in H^1(0,T;L^2(\Omega)^2)$ and
\begin{equation}
\label{d3}
\norm{f}_{H^1(0,T;L^2(\Omega)^2)} \leq C \left( \norm{\partial_t G}_{H^{2,1}(Q)^2} +\norm{G}_{H^{2,1}(Q)^2} \right)
\leq C \norm{G}_{H^{4,2}(Q)^2},
\end{equation}
for some constant $C$ depending only on $\Omega$, $T$ and $M$.

Moreover, since the operator $-iH(A,q^\pm,p)$ is m-dissipative in $L^2(\Omega)^2$, by Lemma \ref{lm2}, we deduce from \eqref{d2} upon applying \cite[Lemma 2.1]{CKS} (with $X=L^2(\Omega)^2$, $U=-iH(A,q^\pm,p)$ and $B=0$) that there exists a unique solution $v \in H^{2,1}(Q)^2$ to \eqref{d2}, such that 
$$
\norm{v}_{H^{2,1}(Q)^2} \leq C \norm{f}_{H^1(0,T;L^2(\Omega)^2)}.
$$
Finally, bearing in mind that $u=v+G$, we obtain \eqref{d0} by combining the above estimate with \eqref{d1} and \eqref{d3}.
\end{proof}

Armed with Lemma \ref{lm2} we may now seek higher regularity for the solution to the IBVP \eqref{a1} upon imposing more restrictive conditions on $\Gamma$, $A$, $p$, $q^\pm$, $u_0$ and $g$.

\subsection{Improved regularity and proof of Proposition \ref{pr1}}

\subsubsection{Improved regularity result}

The statement we are aiming for can be formulated as follows.

\begin{lem}
\label{lm4}
Fix $m\in\mathbb{N}$ and assume that $\Gamma$ is $C^{2(m+1)}$. Let $A\in W^{2m+1,\infty}(\Omega,\mathbb{R}^n)$ verify $\nabla \cdot A=0$ a.e. in $\Omega$ and pick
$p\in W^{2m+1,\infty}(\Omega,\mathbb{R})$ and $q^\pm \in W^{2m+1,\infty}(\Omega,\mathbb{R})$ in such a way that
$$ \norm{A}_{W^{2m+1,\infty}(\Omega)^n}+ \norm{p}_{W^{2m+1,\infty}(\Omega)}+\norm{q^+}_{ W^{2m+1,\infty}(\Omega)}+ \norm{q^-}_{W^{2m+1,\infty}(\Omega)} \leq M,$$
for some {\it a priori} fixed positive constant $M$. Then for all
$g=(g^+,g^-)^T \in H^{2(m+7/4),m+7/4}(\Sigma)^2$ and all $u_0=(u_0^+,u_0^-)^T \in H^{2m+3}(\Omega)^2$ fulfilling the compatibility conditions 
$$
\partial_t^\ell g(\cdot,0)=( -i )^\ell \mathcal{H}(A,q^\pm,p)^\ell u_0\ \mbox{on}\ \Gamma,\ \ell=0,\cdots,m,
$$
there exists a unique solution $u \in \bigcap\limits_{\ell=0}^{m+1} H^{m+1-\ell}(0,T;H^{2\ell}(\Omega)^2)$ to \eqref{d2}. Moreover $u$ satisfies the estimate 
\begin{equation}
\label{d4a}
\sum_{\ell=0}^{m+1} \norm{u}_{H^{m+1-\ell}(0,T;H^{2\ell}(\Omega)^2)} \leq C \left( \norm{u_0}_{H^{2m+3}(\Omega)^2}+\norm{g}_{H^{2(m+7/4),m+7/4}(\Sigma)^2} \right),
\end{equation}
where $C$ is a positive constant depending only on $\Omega$, $T$ and $M$.
\end{lem}

\begin{proof}
We will prove the result by induction on $m \in \N$.

\noindent 1) {\it Base case}. We first consider the case $m=1$. Put $z=(z^+,z^-)=\partial_t u$, where $u=(u^+,u^-)^T \in H^{2,1}(Q)^2$ is the solution to \eqref{a1} given by Lemma \ref{lm3}. Then upon differentiating \eqref{a1b} with respect to $t$, we get that
\begin{equation}
\label{d4b}
\left\{
\begin{array}{ll}
-i\partial_t z + \mathcal{H}(A,q^\pm,p) z =0 & \textrm{in}\ Q, \\
z=\partial_t g &\textrm{on}\ \Sigma,\\
z(\cdot,0)=z_0 &\textrm{in}\ \Omega,
\end{array}
\right.
\end{equation}
with $z_0=-i(-\Delta+\tilde{q}+\tilde{A} \cdot \nabla+\tilde{p})u_0 \in H^3(\Omega)^2$. Since
$\partial_t g(\cdot,0)=z_0$ on $\Gamma$, by \eqref{d4}, and since $\partial_t g \in H^{7/2,7/4}(\Sigma)^2$ with $\norm{\partial_t g}_{H^{7/2,7/4}(\Sigma)^2} \leq  \norm{g}_{H^{11/2,11/4}(\Sigma)^2}$ from \eqref{n1}-\eqref{n2}, then Lemma \ref{lm3} yields $z \in H^{2,1}(Q)^2$ and
\begin{eqnarray}
\norm{z}_{H^{2,1}(Q)^2} &\leq & C \left( \norm{z_0}_{H^3(\Omega)^2}+ \norm{\partial_t g}_{H^{7/2,7/4}(\Sigma)^2} \right) \nonumber \\
&\leq & C \left( \norm{u_0}_{H^5(\Omega)^2} + \norm{g}_{H^{11/2,11/4}(\Sigma)^2} \right). \label{d5}
\end{eqnarray}
Therefore we have obtained $\partial_t^\ell u \in H^{2,1}(Q)^2$ for $\ell=0,1$, whence
\begin{equation}
\label{d5b}
u \in \cap_{\ell=0}^1 H^{2-\ell}(0,T;H^{2\ell}(\Omega)^2).
\end{equation}
Moreover, we derive from the two estimates \eqref{d0} and \eqref{d5} that
\begin{equation}
\label{d5c}
\sum_{\ell=0}^1 \norm{u}_{H^{2-\ell}(0,T;H^{2\ell}(\Omega)^2)}
\leq C \left(\norm{u_0}_{H^5(\Omega)^2}+\ \norm{g}_{H^{11/2,11/4}(\Sigma)^2} \right).
\end{equation}
It remains to show that $u \in L^2(0,T;H^4(\Omega)^2)$ verifies $\norm{u}_{L^2(0,T;H^4(\Omega)^2)} \leq C  \left(\norm{u_0}_{H^5(\Omega)^2}+\ \norm{g}_{H^{11/2,11/4}(\Sigma)^2} \right)$. This can be done by applying the elliptic regularity theorem twice. Indeed, for a.e. $t \in (0,T)$ we infer from the IBVP \eqref{a1b} that $u(\cdot,t)$ is solution to the following elliptic system
\begin{equation}
\label{d6}
 \left\{
\begin{array}{ll}
\Delta u(\cdot,t)=h(\cdot,t) & \textrm{in}\ \Omega, \\
u(\cdot,t)=g(\cdot,t) &\textrm{on}\ \Gamma,
\end{array}
\right.
\end{equation}
where $h= -iz - \tilde{q} u  - \tilde{A} \cdot \nabla u+ \tilde{p} u$. As $h(\cdot,t) \in H^1(\Omega)^2$ and $g(\cdot,t)\in H^{11/2}(\Gamma)^2  \subset H^{5/2}(\Gamma)^2$, then we have $u(\cdot,t)\in H^3(\Omega)^2$ by elliptic regularity, with
\begin{eqnarray}
\norm{u(\cdot,t)}_{H^3(\Omega)^2} &\leq & C \left( \norm{h(\cdot,t)}_{H^1(\Omega)^2} + \norm{g(\cdot,t)}_{H^{5/2}(\Gamma)^2} \right) \nonumber \\
& \leq & C \left( \norm{z(\cdot,t)}_{H^1(\Omega)^2}+\norm{u(\cdot,t)}_{H^2(\Omega)^2}+\norm{g(\cdot,t)}_{H^{11/2}(\Gamma)^2} \right). \label{d7}
\end{eqnarray}
As a consequence we have $h(\cdot,t)\in H^2(\Omega)^2$ for a.e. $t\in(0,T)$, and since $g(\cdot,t) \in H^{11/2}(\Gamma)^2  \subset H^{7/2}(\Gamma)^2$, then the elliptic regularity theorem entails $u(\cdot,t) \in H^4(\Omega)^2$ and
\begin{eqnarray*}
\norm{u(\cdot,t)}_{H^4(\Omega)^2} & \leq & C \left( \norm{h(\cdot,t)}_{H^2(\Omega)^2}+\norm{g(\cdot,t)}_{H^{7/2}(\Gamma)^2} \right) \nonumber \\
& \leq & C \left( \norm{z(\cdot,t)}_{H^2(\Omega)^2}+ \norm{u(\cdot,t)}_{H^3(\Omega)^2}+ \norm{g(\cdot,t)}_{H^{11/2}(\Gamma)^2} \right).
\end{eqnarray*} 
Putting this together with \eqref{d7} we obtain that
\begin{equation}
\label{d8}
\norm{u(\cdot,t)}_{H^4(\Omega)^2} \leq C \left( \norm{z(\cdot,t)}_{H^2(\Omega)^2}+ \norm{u(\cdot,t)}_{H^2(\Omega)^2}+ \norm{g(\cdot,t)}_{H^{11/2}(\Gamma)^2} \right).
\end{equation} 
Next, since $u$ and $z$ are both in $L^2(0,T;H^2(\Omega)^2)$ and since $g \in L^2(0,T; H^{11/2}(\Gamma)^2)$, then \eqref{d8} yields $u \in L^2(0,T;H^4(\Omega)^2)$ and 
$$ \norm{u}_{L^2(0,T;H^4(\Omega)^2)} \leq C \left( \norm{u}_{H^1(0,T;H^2(\Omega)^2)}+ \norm{g}_{L^2(0,T;H^{11/2}(\Gamma)^2)} \right). $$
In light of \eqref{d5c} this entails that
$$ \norm{u}_{L^2(0,T;H^4(\Omega)^2)} \leq C \left( \norm{u_0}_{H^5(\Omega)^2}+ \norm{g}_{H^{11/2,11/4}(\Sigma)^2} \right). $$
Now the result for $m=1$ follows from this and \eqref{d5b}-\eqref{d5c}.\\

\noindent 2) {\it Induction step}. Let us suppose that the claim of Lemma \ref{lm4} holds for some $m\in\mathbb{N}$. We shall prove that it is still true for $m+1$, provided $\Gamma$ is $C^{2m+3}$, $A \in W^{2m+3,\infty}(\Omega,\R^n)$, $q^\pm \in W^{2m+3,\infty}(\Omega,\R)$, $p \in W^{2m+3,\infty}(\Omega,\R)$ and 
$(g,u_0) \in H^{2(m+11 \slash 4),m+11/4}(\Sigma)^2 \times H^{2m+5}(\Omega)^2$ verify the compatibility condition \eqref{d4} where $m+1$ is substituted for $m$, i.e.
\begin{equation}
\label{d9}
\partial_t^\ell g(\cdot,0)=( -i )^\ell \mathcal{H}(A,q^\pm,p)^\ell u_0\ \mbox{on}\ \Gamma,\ \ell=0,\cdots,m+1,
\end{equation}

Let $u$ denote the $\bigcap_{\ell=0}^{m+1}  H^{m+1-\ell} \left( 0,T;H^{2\ell}(\Omega)^2 \right)$-solution to \eqref{a1b} satisfying \eqref{d4a}. Then in a similar way to the base case, we differentiate the system \eqref{a1b} with respect to the time variable and get that $z=\partial_t u$ solves \eqref{d4b} with 
$z_0=-i \mathcal{H}(A,q^\pm,p) u_0 \in H^{2m+3}(\Omega)^2$. Moreover, we know from \eqref{n1} that $\partial_t g \in H^{2(m+7 \slash 4), m+ 7 \slash 4}(\Sigma)^2$ and from \eqref{d9} that
\begin{eqnarray*}
\partial_t^\ell (\partial_t g)(\cdot,0) & = & \partial_t^{\ell+1} g(\cdot,0) = (-i)^{\ell+1} \mathcal{H}(A,q^\pm,p)^{\ell+1} u_0 \\
& = & (-i)^\ell \mathcal{H}(A,q^\pm,p)^{\ell} z_0,\ \ell=0,\ldots,m.
\end{eqnarray*}
Therefore, we have $z \in \bigcap_{\ell=0}^{m+1}  H^{m+1-\ell} \left( 0,T;H^{2\ell}(\Omega)^2 \right)$ and
\begin{eqnarray}
\sum_{\ell=0}^{m+1} \norm{z}_{H^{m+1-\ell}(0,T;H^{2\ell}(\Omega)^2)} & \leq & C \left( \norm{z_0}_{H^{2m+3}(\Omega)^2}+\norm{\partial_t g}_{H^{2(m+7/4),m+7/4}(\Sigma)^2} \right) \nonumber \\
& \leq & C \left( \norm{u_0}_{H^{2m+5}(\Omega)^2}+\norm{g}_{H^{2(m+11/4),m+11/4}(\Sigma)^2} \right), \label{d9b}
\end{eqnarray}
by induction hypothesis (and upon using the estimate $\norm{\partial_t g}_{H^{2(m+7/4),m+7/4}(\Sigma)^2} \leq  \norm{g}_{H^{2(m+11 \slash 4), m + 11 \slash 4}(\Sigma)^2}$, arising from  \eqref{n2}).
Since $u \in \bigcap_{\ell=0}^{m+1}  H^{m+1-\ell} \left( 0,T;H^{2\ell}(\Omega)^2 \right)$ then this may be equivalently rewritten as
\begin{equation}
\label{d10}
u \in \bigcap_{\ell=0}^{m+1}  H^{m+2-\ell} (0,T;H^{2\ell}(\Omega)^2)
\end{equation}
and we infer from \eqref{d4a} and \eqref{d9b} that
\begin{equation}
\label{d11}
\sum_{\ell=0}^{m+1} \norm{u}_{H^{m+2-\ell}(0,T;H^{2\ell}(\Omega)^2)} \leq C \left( \norm{u_0}_{H^{2m+5}(\Omega)^2}+\norm{g}_{H^{2(m+11/4),m+11/4}(\Sigma)^2} \right). 
\end{equation}
Thus we are left with the task of showing that $u \in L^2(0,T;H^{2(m+2)}(\Omega)^2)$ and that
$$ \norm{u}_{L^2(0,T;H^{2(m+2)}(\Omega)^2)} \leq C\left( \norm{u_0}_{H^{2m+5}(\Omega)^2}+\norm{g}_{H^{2(m+11/4),m+11/4}(\Sigma)^2} \right).$$
This can be done with the elliptic regularity theorem upon using for a.e. $t \in (0,T)$ that $u(\cdot,t)$ is solution to \eqref{d6} with $h(\cdot,t)= -iz(\cdot,t)- \tilde{q} u(\cdot,t)  - \tilde{A} \cdot \nabla u(\cdot,t)+ \tilde{p} u(\cdot,t) \in H^{2m+1}(\Omega)^2$ and $g(\cdot,t) \in H^{2(m+11 \slash 4)}(\Gamma)^2 \subset H^{2m+5 \slash 2}(\Gamma)^2$. We get that $u(\cdot,t)\in H^{2m+3}(\Omega)^2$ and
\begin{eqnarray}
\norm{u(\cdot,t)}_{H^{2m+3}(\Omega)^2} &\leq & C \left( \norm{h(\cdot,t)}_{H^{2m+1}(\Omega)^2} + \norm{g(\cdot,t)}_{H^{2m+5 \slash 2}(\Gamma)^2} \right) \nonumber \\
& \leq & C \left( \norm{z(\cdot,t)}_{H^{2m+1}(\Omega)^2}+\norm{u(\cdot,t)}_{H^{2m+2}(\Omega)^2}+\norm{g(\cdot,t)}_{H^{2(m+11 \slash 4)}(\Gamma)^2} \right). \label{d12}
\end{eqnarray}
Therefore we have $h(\cdot,t)\in H^{2(m+1)}(\Omega)^2$ for a.e. $t \in (0,T)$, and since $g(\cdot,t) \in H^{2(m+11 \slash 4)}(\Gamma)^2 \subset H^{2m+7 \slash 2}(\Gamma)^2$, we obtain by elliptic regularity that $u(\cdot,t) \in H^{2(m+2)}(\Omega)^2$ and
\begin{eqnarray*}
\norm{u(\cdot,t)}_{H^{2(m+2)}(\Omega)^2} & \leq & C \left( \norm{h(\cdot,t)}_{H^{2(m+1)}(\Omega)^2}+\norm{g(\cdot,t)}_{H^{2m+7 \slash 2}(\Gamma)^2} \right) \nonumber \\
& \leq & C \left( \norm{z(\cdot,t)}_{H^{2(m+1)}(\Omega)^2}+ \norm{u(\cdot,t)}_{H^{2m+3}(\Omega)^2}+ \norm{g(\cdot,t)}_{H^{2(m+ 11 \slash 4)}(\Gamma)^2} \right).
\end{eqnarray*}
Putting this last estimate with \eqref{d12}, we end up getting that 
\begin{equation}
\label{d13} 
\norm{u(\cdot,t)}_{H^{2(m+2)}(\Omega)^2} \leq C \left( \norm{z(\cdot,t)}_{H^{2(m+1)}(\Omega)^2}+ \norm{u(\cdot,t)}_{H^{2(m+1)}(\Omega)^2}+ \norm{g(\cdot,t)}_{H^{2(m+11 \slash 4)}(\Gamma)^2} \right).
\end{equation} 
Further, since $u$ and $z$ are in $L^2(0,T;H^{2(m+1)}(\Omega)^2)$ and since $g \in L^2(0,T; H^{2(m+11 \slash 4)}(\Gamma)^2)$, we infer from \eqref{d13} that
$u \in L^2(0,T;H^{2(m+2)}(\Omega)^2)$ verifies
$$
\norm{u}_{L^2(0,T;H^{2(m+2)}(\Omega)^2)}  \leq  C \left( \norm{u}_{H^1(0,T;H^{2(m+1)}(\Omega)^2)} + \norm{g}_{L^2(0,T;H^{2(m+11/4)}(\Gamma)^2)} \right).
$$ 
In view of \eqref{d11} this entails that
$$
\norm{u}_{L^2(0,T;H^{2(m+2)}(\Omega)^2)}  \leq C \left( \norm{u_0}_{H^{2m+5}(\Omega)^2}+ \norm{g}_{H^{2(m+11/4),m+11/4}(\Sigma)^2} \right), 
$$
which, together with \eqref{d10}-\eqref{d11} yield the statement of Lemma \ref{lm4} for $m+1$ .
\end{proof}

Having established Lemma \ref{lm4}, we turn now to proving Proposition \ref{pr1}. 

\subsubsection{Proof of Proposition \ref{pr1}}
We apply Lemma \ref{lm4} with $m=N$ and get a unique solution $u$ to \eqref{a1} within the space $H^2(0,T;H^{2(N-1)}(\Omega)^2)$. Since $2(N-1) > \frac{n}{2} +1$ from the very definition of $N$, then $u \in W^{1,\infty}(0,T;W^{1,\infty}(\Omega)^2)$ by the Sobolev embedding theorem. Moreover, \eqref{d4a} yields
\begin{eqnarray*}
\norm{u}_{W^{1,\infty}(0,T;W^{1,\infty}(\Omega)^2)} & \leq & \norm{u}_{H^2(0,T;H^{2(N-1)}(\Omega)^2)} \\
& \leq & C \left( \norm{u_0}_{H^{2N+3}(\Omega)^2}+\norm{g}_{H^{2(N+7/4),N+7/4}(\Sigma)^2} \right),
\end{eqnarray*}
for some constant $C$ depending only on $\Omega$, $T$ and $M$. This proves the desired result.



\section{Analysis of the inverse problem}
\label{sec-ip}

This section contains the proof of Theorem \ref{th1}.

\subsection{Preliminaries: Carleman estimate and all that}

In this section, we establish in Corollary \ref{cor1} a weighted energy estimate for the Schr\"odinger equation, which the main tool used in the derivation of Theorem \ref{th1}. This inequality is a byproduct of the global Carleman estimate for the Schr\"odinger operator 
of \cite[Proposition 1]{BP}, that we recall in Proposition \ref{pr2}. 
In order to state this inequality we consider a function
$\tilde{\beta} \in C^4(\overline{\Omega},\R_+)$ and an open subset $\Gamma_* \subset \Gamma$ fulfilling the following conditions:
\begin{Assumption} \hfill \break \vspace*{-.4cm}
\label{c1}
\begin{enumerate}[(1)]
\item There exists a constant $c>0$ such that the estimate $\abs{\nabla \tilde{\beta}(x)}\geq c$ holds for all 
$x\in\Omega$;
\item $\partial_\nu\tilde{\beta}(x)=\nabla\tilde{\beta}(x)\cdot\nu(x)<0$ for all $x \in \partial \Omega \setminus \Gamma_*$, where $\nu$ is the outward unit normal vector to $\Gamma$;
\item There exists $\Lambda_1>0$ and $\epsilon>0$ such that $\lambda\vert\nabla\tilde{\beta}(x)\cdot\zeta\vert^2+D^2\tilde{\beta}(x,\zeta,\zeta)\geq\epsilon\vert\zeta\vert^2$ for all $\zeta\in\mathbb{R}^{n}$, $x\in\Omega$ and $\lambda>\Lambda_1$, where $D^2\tilde{\beta}(x)= \left( \frac{\partial^2\tilde{\beta}(x)}{\partial x_i\partial x_j} \right)_{1\leq i,j\leq n}$ and $D^2\tilde{\beta}(x,\zeta,\zeta)$ denotes the $\mathbb{R}^{n}$-scalar product of $D^2\tilde{\beta}(x)\zeta$ with $\zeta$.
\end{enumerate}
\end{Assumption}

\begin{remark}
We stress out that there exist actual $\tilde{\beta}$ and $\Gamma_*$ satisfying Assumption \ref{c1}. As a matter of fact, for all $x_0\in\mathbb{R}^{n}\setminus\overline{\Omega}$ fixed, this the case of the function $\tilde{\beta}(x)=\vert x-x_0\vert^2$ and any open subset $\Gamma_* \subset \Gamma$ containing $\{x\in \Gamma;\ (x-x_0)\cdot\nu(x)\geq 0\}$.
\end{remark}

Next, we put 
\begin{equation}
\label{c2}
\beta(x)=\tilde{\beta}(x)+r \Vert\tilde{\beta}\Vert_{L^\infty(\Omega)},\ x \in \overline{\Omega},
\end{equation}
for some $r >1$ and $K=\norm{\beta}_{L^\infty(\Omega)}$, and we set 
\begin{equation}
\label{c3}
\varphi(x,t)=\frac{e^{2\lambda\beta(x)}}{(T+t)(T-t)}\ \mbox{and}\ \eta(x,t)=\frac{e^{2\lambda K}-e^{\lambda\beta(x)}}{(T+t)(T-t)},\ (x,t)\in \tilde{Q}=\Omega \times (-T,T),
\end{equation}
for some $\lambda>0$. Further, for all $s>0$, we introduce the two following operators acting in $(C^\infty_0)^\prime(\tilde{Q}):$
\begin{equation}
\label{c4}
M_1=i\partial_t+\Delta+s^2\vert\nabla\eta\vert^2\ \mbox{and}\ M_2=is\eta^\prime+2s\nabla\eta\cdot\nabla+s(\Delta\eta). 
\end{equation}
It can be checked that $M_1$ (resp., $M_2$) is the adjoint (resp., skew-adjoint) part of the operator $e^{-s \eta} L e^{s \eta}$, where
$L=i\partial_t+\Delta$. Then the global Carleman estimate borrowed from \cite[Proposition 1]{BP} is as follows.

\begin{proposition}
\label{pr2}
Let $\tilde{\beta}$ and $\Gamma_*$ fufill Assumption \ref{c1}, let $\beta$, $\varphi$ and $\eta$ be given by \eqref{c2}-\eqref{c3}, and let the operators $M_j$, $j=1,2$, be defined by \eqref{c4}. Then there are two constants $s_0>1$ and $C_0>0$, depending only on $\Omega$, $T$ and $\Gamma_*$, such that the estimate
\begin{eqnarray*}
& & s\Vert e^{-s\eta}\nabla w\Vert^2_{L^2(\tilde{Q})^n}+s^3\Vert e^{-s\eta}w\Vert^2_{L^2(\tilde{Q})}+\sum\limits_{j=1,2}\Vert M_je^{-s\eta}w\Vert^2_{L^2(\tilde{Q})} \\
&\leq & C_0 \left( s\Vert e^{-s\eta}\varphi^{1/2}(\partial_\nu\beta)^{1/2}\partial_\nu w\Vert^2_{L^2(\tilde{\Sigma}_*)}+\Vert e^{-s\eta}Lw\Vert^2_{L^2(\tilde{Q})} \right),
\end{eqnarray*} 
holds for all $s > s_0$ and all $w\in L^2(-T,T;H^1_0(\Omega))$ satisfying $L w\in L^2(\tilde{Q})$ and $\partial_\nu w\in L^2(\tilde{\Sigma_*})$, where $\tilde{\Sigma}_*= \Gamma_* \times (-T,T)$.
\end{proposition}

As a corollary we have the following technical result. Its proof can be found in \cite[Section 4.1]{HKSY}, but for the sake of self-containedness and for the convenience of the reader, we give it at the end of the section. 

\begin{corollary}
\label{cor1}
Under the conditions of Proposition \ref{pr2}, we have
\begin{eqnarray*}
& & s^{-1/2} \left( \norm{e^{-s\eta} z}^2_{L^2(\tilde{Q})} + \norm{e^{-s\eta}\nabla z}^2_{L^2(\tilde{Q})^n} \right)+ \norm{e^{-s\eta(\cdot,0)} z(\cdot,0)}^2_{L^2(\Omega)} \\
& \leq & C_0 s^{-3/2} \left( s \norm{e^{-s\eta}\varphi^{1/2}(\partial_\nu\beta)^{1/2}\partial_\nu z}^2_{L^2(\tilde{\Sigma}_*)}
+ \norm{e^{-s\eta}L z}^2_{L^2(\tilde{Q})} \right),\ s>s_0,
\end{eqnarray*}
for all $w\in L^2(-T,T;H^1_0(\Omega))$ fulfilling $L w\in L^2(\tilde{Q})$ and $\partial_\nu w\in L^2(\tilde{\Sigma}_*)$,  where $C_0$ and $s_0$ are the same as in Proposition \ref{pr2}.
\end{corollary}
\begin{proof}
Put $w=e^{-s\eta} z$. Since $\lim\limits_{t\longrightarrow -T}\eta(x,t)=+\infty$ for all $x\in\Omega$ then $\lim\limits_{t\longrightarrow -T}w(\cdot,t)=0$ in $L^2(\Omega)$ and hence
$$
\Vert w(\cdot,0)\Vert^2_{L^2(\Omega)}=\displaystyle\int_{(-T,0)\times\Omega} \partial_t \vert w(x,t) \vert^2 dx dt
=2 \Re \left(\displaystyle\int_{(-T,0)\times\Omega} (\partial_t w) \overline{w}(x,t)dx dt\right).
$$
On the other hand we have
\begin{eqnarray*}
& & \Im \left(\displaystyle\int_{(-T,0)\times\Omega} (M_1w)\overline{w}(x,t)dxdt \right) \\
& = & \Re \left( \int_{(-T,0)\times\Omega} (\partial_t w)\overline{w}(x,t)dxdt\right)+\Im\left(\displaystyle\int_{(-T,0)\times\Omega}(\Delta w +s^2\vert\nabla\eta\vert^2 w) \overline{w}(x,t)dxdt \right) \\
&= & \Re\left(\int_{(-T,0)\times\Omega} (\partial_t w) \overline{w}(x,t) dxdt \right)- \Im \left( \int_{(-T,0)\times\Omega}(\vert\nabla w\vert^2-s^2\vert\nabla\eta\vert^2\vert w\vert^2)(x,t)dxdt\right) \\
&= & \Re\left(\int_{(-T,0)\times\Omega} (\partial_t w) \overline{w}(x,t) dxdt \right),
\end{eqnarray*}
whence
$\Vert w(\cdot,0)\Vert^2_{L^2(\Omega)}=2 \Im \left(\displaystyle\int_{(-T,0)\times\Omega} (M_1w) \overline{w}(x,t)dxdt \right)$. Therefore we get
$$
\Vert e^{-s\eta(\cdot,0)} z(\cdot,0)\Vert^2_{L^2(\Omega)} \leq 2\Vert M_1 w\Vert_{L^2(\tilde{Q})}\Vert w\Vert_{L^2(\tilde{Q})}\leq s^{-3/2} \left( s^3\Vert e^{-s\eta} z \Vert^2_{L^2(\tilde{Q})}+\Vert M_1e^{-s\eta} z\Vert^2_{L^2(\tilde{Q})}\right)
$$
with the help of the Cauchy-Schwarz and H\"older inequalities. As a consequence we have
\begin{eqnarray*}
& & s^{-1/2} \left( \norm{e^{-s\eta} z}^2_{L^2(\tilde{Q})} +  \norm{e^{-s\eta} \nabla z}^2_{L^2(\tilde{Q})^n} \right)+ \norm{e^{-s\eta(\cdot,0)} z(\cdot,0)}^2_{L^2(\Omega)} \\
& \leq & s^{-3/2}\left( s\Vert e^{-s\eta}\nabla z \Vert^2_{L^2(\tilde{Q})^n} + s^3\Vert e^{-s\eta} z \Vert^2_{L^2(\tilde{Q})}
+\Vert M_1e^{-s\eta} z \Vert^2_{L^2(\tilde{Q})} \right)\\
& \leq &  C_0 s^{-3/2} \left(s\Vert e^{-s\eta}\varphi^{1/2}(\partial_\nu\beta)^{1/2}\partial_\nu\ z \Vert^2_{L^2(\tilde{\Sigma}_*)}+\Vert e^{-s\eta} Lz \Vert^2_{L^2(\tilde{Q})} \right),\ s>s_0,
\end{eqnarray*}
by Proposition \ref{pr2}, which is the desired result.
\end{proof}

\subsection{Proof of Theorem \ref{th1}}
We start by linearizing the system \eqref{a1}.
That is, we denote by $u_j=(u_j^+,u_j^-)$, $j=1,2$, the solution to \eqref{a1}, where $(A_j,p_j,q_j^\pm)$ is substituted for $(A,p,q^\pm)$ and we take the difference of the two systems \eqref{a1} associated with $j=1,2$. Thus, putting $A=A_1-A_2$, $p=p_1-p_2$ and $q^\pm=q_1^\pm-q_2^\pm$, we get that $u=(u^+,u^-)^T$, where $u^\pm=u_1^\pm-u_2^\pm$, solves
\begin{equation}
\label{p1}
 \left\{
\begin{array}{ll}
-i\partial_t u^+-\Delta u^++q^+_1u^+=-A_1\cdot\nabla u^- - p_1u^- - A\cdot\nabla u_2^- - pu_2^- - q^+u_2^+ &\textrm{in}\ Q, \\
-i\partial_t u^- -\Delta u^-+q^-_1u^-=A_1\cdot\nabla u^+ - p_1u^+ + A\cdot\nabla u_2^+ - pu_2^+ - q^-u_2^- &\textrm{in}\ Q, \\
u^+(\cdot,0)=0,\ u^-(\cdot,0)=0&\textrm{in}\ \Omega,\\
u^+=0,\ u^-=0&\textrm{on}\ \Sigma.
\end{array}
\right.
\end{equation}
Since $u^\pm \in H^2(0,T;L^2(\Omega)) \cap H^1(0,T;H^2(\Omega)\cap H^1_0(\Omega))$, we can differentiate \eqref{p1} with respect to the time-variable: We obtain that $v=\partial_t u=(v^+,v_-)$, where $v^\pm=\partial_t u^\pm\in H^1(0,T;L^2(\Omega))\cap L^2(0,T;H^2(\Omega)\cap H^1_0(\Omega))$, is solution to the following coupled system
\begin{equation}
\label{p2}
\left\{
\begin{array}{ll}
-i\partial_tv^+-\Delta v^++q^+_1v^+=-A_1\cdot\nabla v^- - p_1 v^- - A \cdot \nabla \partial_t u_2^- - p \partial_t u_2^- - q^+\partial_tu_2^+&\textrm{in}\ Q, \\
-i\partial_tv^- - \Delta v^-+q^-_1v^-=A_1\cdot\nabla v^+ - p_1 v^+ + A \cdot \nabla \partial_tu_2^+ - p \partial_t u_2^+ - q^-\partial_t u_2^-&\textrm{in}\ Q, \\
v^+(\cdot,0)=-i (A \cdot\nabla u_0^- + p u_0^- + q^+u_0^+) &\textrm{in}\ \Omega,\\
v^-(\cdot,0)=-i (-A\cdot\nabla u_0^+ + pu_0^+ + q^-u_0^-) &\textrm{in}\ \Omega,\\
v^+=0,\ v^-=0&\textrm{on}\ \Sigma.
\end{array}
\right.
\end{equation}
We extend $u^\pm_2$ on $\tilde{Q}$ by setting $u^\pm_2(x,t)=\overline{u^\pm_2(x,-t)}$ for a.e. $(x,t)\in \Omega \times (-T,0)$. Since $u_0^{\pm}$, $A$, $p$ and $q^\pm$ are all real valued, then it is easy to see that the function $v$, where $v^\pm$ is extended on $\Omega \times (-T,0)$ as $v^\pm(x,t)=-\overline{v^\pm(x,-t)}$, is solution to
\begin{equation}
\label{p3}
\left\{
\begin{array}{ll}
-i\partial_tv^+ - \Delta v^+ + q^+_1 v^+ = -A_1\cdot\nabla v^- - p_1 v^- - A\cdot\nabla\partial_t u_2^- - p\partial_tu_2^- - q^+\partial_tu_2^+ &\textrm{in}\ \tilde{Q}, \\
-i\partial_tv^--\Delta v^-+q^-_1v^-= A_1\cdot\nabla v^+ - p_1v^+ + A\cdot\nabla\partial_tu_2^+  - p\partial_tu_2^+ - q^-\partial_tu_2^-&\textrm{in}\ \tilde{Q}, \\
v^+(\cdot,0)=-i (A \cdot\nabla u_0^- + p u_0^- + q^+u_0^+) &\textrm{in}\ \Omega,\\
v^-(\cdot,0)= -i (-A\cdot\nabla u_0^+ + pu_0^+ + q^-u_0^-) &\textrm{in}\ \Omega,\\
v^+=0,\ v^-=0&\textrm{on}\ \tilde{\Sigma},
\end{array}
\right.
\end{equation}
with $\tilde{\Sigma}=\Gamma \times (-T,T)$.

Let us apply Corollary \ref{cor1} to $v^{\pm}$, where for the sake of notational simplicity we write $\mu^\pm$ instead of \\
$\Vert e^{-s\eta}\varphi^{1/2}(\partial_\nu\beta)^{1/2}\partial_\nu v^\pm \Vert^2_{L^2(\tilde{\Sigma}_*)}$.
Then, in light of \eqref{p3}, we get for all $s>s_0$ that
\begin{eqnarray*}
& & s^{-1/2} \left( \norm{e^{-s\eta} v^\pm}^2_{L^2(\tilde{Q})} + \norm{e^{-s\eta} \nabla v^\pm}^2_{L^2(\tilde{Q})^n} \right)+ \norm{e^{-s\eta(\cdot,0)}(\pm A\cdot\nabla u_0^\mp + p u_0^\mp +q^\pm u_0^\pm)}^2_{L^2(\Omega)} \nonumber \\
& \leq & C_0 s^{-3/2} \left( \Vert e^{-s\eta}A\cdot \nabla \partial_tu_2^\mp \Vert^2_{L^2(\tilde{Q})} +\Vert e^{-s\eta}p\partial_tu_2^\mp \Vert^2_{L^2(\tilde{Q})} +\Vert e^{-s\eta}q^\pm \partial_t u_2^\pm \Vert^2_{L^2(\tilde{Q})} \right. \\
& & \hspace*{1.5cm}  \left.   +\Vert e^{-s\eta} A_1\cdot\nabla v^\mp \Vert^2_{L^2(\tilde{Q})}+\Vert e^{-s\eta} p_1 v^\mp \Vert^2_{L^2(\tilde{Q})}+ s \mu^\pm \right) \\
& \leq & C_1 s^{-3/2} \left( \Vert e^{-s\eta}A\Vert^2_{L^2(\Omega)^n} +\Vert e^{-s\eta}p\Vert^2_{L^2(\Omega)} + \Vert e^{-s\eta}q^\pm \Vert^2_{L^2(\tilde{Q})}+\Vert e^{-s\eta}v^\mp\Vert^2_{L^2(\tilde{Q})} +\Vert e^{-s\eta}\nabla v^\mp \Vert^2_{L^2(\tilde{Q})} + s \mu^\pm \right),
\end{eqnarray*}
with $C_1=C_0 \max(M^2,C^2)$, where $C$ is the constant appearing in Proposition \ref{pr1}.
Indeed, in the last line we used the energy inequality \eqref{d4c}, entailing that $\norm{\partial_t u_2^\pm}_{L^\infty(\tilde{Q})} + \norm{\nabla \partial_t u_2^\pm}_{L^\infty(\tilde{Q})^n} \leq C$.
Summing up the two above estimates and recalling that $e^{-s\eta(x,t)}\leq e^{-s\eta(x,0)}$ for all $(x,t) \in \tilde{Q}$ then leads for all  $s > s_1=\max(s_0, 2C_1)$ to
\begin{eqnarray}
& & \Vert e^{-s\eta(\cdot,0)}(-A\cdot\nabla u_0^+ +pu_0^+ + q^-u_0^-)\Vert^2_{L^2(\Omega)}+\Vert e^{-s\eta(\cdot,0)}(A\cdot\nabla u_0^-+pu_0^- + q^+u_0^+)\Vert^2_{L^2(\Omega)} \nonumber \\
& \leq & C_1'  \left( 
s^{-3 /2} \left( \Vert e^{-s\eta(\cdot,0)}A\Vert^2_{L^2(\Omega)^n} + \Vert e^{-s\eta(\cdot,0)}p\Vert^2_{L^2(\Omega)} + \Vert e^{-s\eta(\cdot,0)} q^+ \Vert^2_{L^2(\Omega)} + \Vert e^{-s\eta(\cdot,0)} q^- \Vert^2_{L^2(\Omega)} \right) \right. \nonumber \\
& & \hspace*{.5cm} \left. + s^{-1/2} \left( \Vert e^{-s\eta(\cdot,0)}\varphi^{1/2}(\partial_\nu\beta)^{1/2}\partial_\nu v^-\Vert^2_{L^2(\tilde{\Sigma}_*)} + 
\Vert e^{-s\eta(\cdot,0)}\varphi^{1/2}(\partial_\nu\beta)^{1/2}\partial_\nu v^+\Vert^2_{L^2(\tilde{\Sigma}_*)} \right) \right), \label{p5}
\end{eqnarray}
where $C_1'=2C_1$.

Having established \eqref{p5}, we shall now specify $u_0^\pm$ in order to estimate $A$, $p$ and $q^\pm$. Namely we probe the system \eqref{a1} with $n+2$ initial states $u_0^k=(u_0^{+,k},u_0^{-,k})$, $k=1,\ldots,n+2$, that we shall describe below, and suitable Dirichlet boundary conditions $g^k=(g^{+,k},g^{-,k})$ fulfilling the compatibility condition \eqref{cck}. We proceed in two steps: In the first one we describe $u_0^k$ for $k=1,2$, while the initial states $u_0^k$ associated with $k=3,\ldots,n+2$, are defined in the second one. In the sequel, $u^k=(u^{+,k},u^{-,k})$ denotes the solution to \eqref{a1} associated with $(u_0^{\pm},g^\pm)=(u_0^{\pm,k},g^{\pm,k})$, we put $v^{\pm,k}= \partial_t u^{\pm,k}$ and $\mu_k^{\pm}(s)= \norm{e^{-s\eta(\cdot,0)}\varphi^{1/2}(\partial_\nu\beta)^{1/2}\partial_\nu v^{\pm,k}}^2_{L^2(\tilde{\Sigma}_*)}$, and we set  $\mu_k=\mu_k^+ + \mu_k^-$. \\

\noindent {\it Step 1.}  For $\alpha>0$ fixed, put $u_0^{+,1}=0$ and let $u_0^{-,1} : \Omega \to \C$ be constant and satisfy $\vert u_0^{-,1}\vert \geq\alpha$ a.e. in $\Omega$. Similarly we set $u_0^{-,2}=0$ and choose a constant function $u_0^{+,2} : \Omega \to \C$ such that $\vert u_0^{+,2} \vert \geq\alpha$. Then, taking $u_0^{\pm}=u_0^{\pm,1}$ in \eqref{p5}, we get for all $s>s_1$ that
\begin{eqnarray*}
& & \Vert e^{-s\eta(\cdot,0)}p\Vert^2_{L^2(\Omega)} + \Vert e^{-s\eta(\cdot,0)}q^-\Vert^2_{L^2(\Omega)} \nonumber \\
& \leq & C_1' s^{-3/2} \left( \Vert e^{-s\eta(\cdot,0)}A\Vert^2_{L^2(\Omega)^n} + \Vert e^{-s\eta(\cdot,0)}p\Vert^2_{L^2(\Omega)}  + \Vert e^{-s\eta(\cdot,0)}q^+\Vert^2_{L^2(\Omega)}+\Vert e^{-s\eta(\cdot,0)}q^- \Vert^2_{L^2(\Omega)}+ s \mu_1\right), 
\end{eqnarray*}
whereas with $u_0^\pm=u_0^{\pm,2}$, we obtain:
 \begin{eqnarray*}
& & \Vert e^{-s\eta(\cdot,0)}q^+\Vert^2_{L^2(\Omega)}+\Vert e^{-s\eta(\cdot,0)}p\Vert^2_{L^2(\Omega)} \nonumber \\
& \leq & C_1' s^{-3/2} \left( \Vert e^{-s\eta(\cdot,0)}A\Vert^2_{L^2(\Omega)^n} + \Vert e^{-s\eta(\cdot,0)}p\Vert^2_{L^2(\Omega)} + \Vert e^{-s\eta(\cdot,0)} q^+ \Vert^2_{L^2(\Omega)} + \Vert e^{-s\eta(\cdot,0)} q^- \Vert^2_{L^2(\Omega)} + s \mu_2 \right).
\end{eqnarray*}
Summing up the two above estimates then yields
\begin{eqnarray*}
& & (1 - 2 C_1' s^{-3/2}) \left( \Vert e^{-s\eta(\cdot,0)}p\Vert^2_{L^2(\Omega)}+\Vert e^{-s\eta(\cdot,0)} q^+ \Vert^2_{L^2(\Omega)} + \Vert e^{-s\eta(\cdot,0)} q^- \Vert^2_{L^2(\Omega)} \right)\\
&  \leq & 2 C_1' \left( s^{-3/2} \Vert e^{-s\eta(\cdot,0)}A\Vert^2_{L^2(\Omega)^n} + s^{-1/2} \sum_{k=1}^2 \mu_k(s) \right),\ s >s_1.
\end{eqnarray*}
Therefore for all $s > s_2=\max(s_1,(4C_1')^{2 \slash 3})$ we have
\begin{equation}
\label{p6}
\Vert e^{-s\eta(\cdot,0)}p\Vert^2_{L^2(\Omega)}+\Vert e^{-s\eta(\cdot,0)} q^+ \Vert^2_{L^2(\Omega)} + \Vert e^{-s\eta(\cdot,0)} q^- \Vert^2_{L^2(\Omega)} \leq C_2 \left( s^{-3/2} \Vert e^{-s\eta(\cdot,0)}A\Vert^2_{L^2(\Omega)} + s^{-1/2} \sum_{k=1}^2 \mu_k(s) \right),
\end{equation}
with $C_2=4C_1'$.\\

\noindent {\it Step 2.} We now choose $2n$ functions $u_0^{\pm,k+2}: \Omega \to \C$, for $k=1,\cdots,n$, such that the two matrices $(U_0^\pm)^*U_0^\pm$, where $U^{\pm}_0=(\partial_l u^{\pm,k+2}_{0})_{1\leq k,l\leq n}$ and $(U_0^{\pm})^*$ is the Hermitian conjugate of $U_0^{\pm}$, are strictly positive definite, i.e.
\begin{equation}
\label{p7}
\exists\ c^{\pm}_0 >0,\ \vert U^{\pm}_0 \xi \vert\geq c^{\pm}_0 \vert \xi \vert,\ \xi \in\mathbb{C}^n.
\end{equation}
For each $k=1,\cdots,n$, we substitute $u_0^{\pm,k+2}$ for $u_0^\pm$ in \eqref{p5} and find that 
$$
\Vert e^{-s\eta(\cdot,0)}(- A\cdot\nabla u_0^{+,k+2} +p u_0^{+,k+2} + q^-u_0^{-,k+2}) \Vert^2_{L^2(\Omega)} +
\Vert e^{-s\eta(\cdot,0)}( A\cdot\nabla u_0^{-,k+2} + p u_0^{-,k+2} + q^+ u_0^{+,k+2} ) \Vert^2_{L^2(\Omega)}
$$
can be upper bounded by
$$
C_1' s^{-3/2} \left( \Vert e^{-s\eta(\cdot,0)} A \Vert^2_{L^2(\Omega)^n} + \Vert e^{-s\eta(\cdot,0)}p \Vert^2_{L^2(\Omega)} + \Vert e^{-s\eta(\cdot,0)}q^+ \Vert^2_{L^2(\Omega)} + \Vert e^{-s\eta(\cdot,0)}q^- \Vert^2_{L^2(\Omega)} +  s \mu_{k+2}(s) \right),
$$
provided $s>s_1$. Thus, taking into account that
$$ \vert \pm A\cdot\nabla u_0^{\mp,k+2} +p u_0^{\mp,k+2} + q^\pm u_0^{\pm,k+2}  \vert^2 \geq \frac{1}{2}\vert A\cdot\nabla u_0^{\mp,k+2} \vert^2-\vert pu_0^{\mp,k+2} + q^\pm u_0^{\pm,k+2} \vert^2,\ k=1,\ldots,n,$$
we get for all $s>s_1$ that
\begin{eqnarray*}
& & \frac{1}{2} \left( \Vert e^{-s\eta(\cdot,0)} A\cdot\nabla u_0^{+,k+2} \Vert^2_{L^2(\Omega)}+\Vert e^{-s\eta(\cdot,0)}A\cdot\nabla u_0^{-,k+2}\Vert^2_{L^2(\Omega)} \right) \\
& \leq & C_1' s^{-3/2} \left( 
\Vert e^{-s\eta(\cdot,0)} A \Vert^2_{L^2(\Omega)^n} +\Vert e^{-s\eta(\cdot,0)}p\Vert^2_{L^2(\Omega)} + \Vert e^{-s\eta(\cdot,0)} q^+ \Vert^2_{L^2(\Omega)} + \Vert e^{-s\eta(\cdot,0)} q^- \Vert^2_{L^2(\Omega)} + s \mu_{k+2}(s) \right) \\
& & + \Vert e^{-s\eta(\cdot,0)}(p u_0^{+,k+2} + q^- u_0^{-,k+2} )\Vert^2_{L^2(\Omega)} + \Vert e^{-s \eta(\cdot,0)}(p u_0^{-,k+2} + q^+ u_0^{+,k+2} )\Vert^2_{L^2(\Omega)},
\end{eqnarray*}
whence
\begin{eqnarray*}
& & \Vert e^{-s\eta(\cdot,0)}A\cdot\nabla u_0^{+,k+2} \Vert^2_{L^2(\Omega)}+\Vert e^{-s\eta(\cdot,0)}A\cdot\nabla u_0^{-,k+2}\Vert^2_{L^2(\Omega)} -C_{1,k} s^{-3/2} \Vert e^{-s\eta(\cdot,0)} A \Vert^2_{L^2(\Omega)^n} \nonumber \\
& \leq & C_{1,k} \left(  \Vert e^{-s\eta(\cdot,0)} p\Vert^2_{L^2(\Omega)} + \Vert e^{-s\eta(\cdot,0)}q^+\Vert^2_{L^2(\Omega)} + \Vert e^{-s\eta(\cdot,0)}q^-\Vert^2_{L^2(\Omega)} + s^{-1/2}  \mu_{k+2}(s) \right),
\end{eqnarray*}
where $C_{1,k}= 2 \left( C_1' + \norm{u_0^{+,k+2}}_{L^\infty(\Omega)}^2 + \norm{u_0^{-,k+2}}_{L^\infty(\Omega)}^2 \right)$.
Summing up over $k=1,\cdots,n$, we obtain for all $s>s_1$ that
\begin{eqnarray*}
& & \Vert e^{-s\eta(\cdot,0)} U_0^+ A\Vert^2_{L^2(\Omega)^n}+\Vert e^{-s\eta(\cdot,0)} U_0^- A \Vert^2_{L^2(\Omega)^n} -
C_1'' s^{-3/2} \Vert e^{-s\eta(\cdot,0)}A\Vert^2_{L^2(\Omega)^n}
\\
&\leq & C_1'' \left( 
 \Vert e^{-s\eta(\cdot,0)} p \Vert^2_{L^2(\Omega)} + \Vert e^{-s\eta(\cdot,0)} q^+ \Vert^2_{L^2(\Omega)} + \Vert e^{-s\eta(\cdot,0)} q^- \Vert^2_{L^2(\Omega)} + s^{-1/2} \sum_{k=1}^n \mu_{k+2}(s) \right),
\end{eqnarray*}
with $C_1''=\sum_{k=1}^n C_{1,k}$. 
Consequently we have
\begin{eqnarray*}
&& (c_0^+ + c_0^- - C_1'' s^{-3/2} ) \Vert e^{-s\eta(\cdot,0)} A \Vert^2_{L^2(\Omega)^n} \\ 
&\leq & C_1'' 
\left( \Vert e^{-s\eta(\cdot,0)} p \Vert^2_{L^2(\Omega)} + \Vert e^{-s\eta(\cdot,0)} q^+ \Vert^2_{L^2(\Omega)} + \Vert e^{-s\eta(\cdot,0)} q^- \Vert^2_{L^2(\Omega)} + s^{-1/2} \sum_{k=1}^n \mu_{k+2}(s) \right),
\end{eqnarray*}
in virtue of \eqref{p7}.
Thus, taking $s>s_2'=\max( s_1, (2 C_1'' \slash (c_0^-+c_0^+))^{2 \slash 3})$ and setting $C_2'=2 C_1''  \slash (c_0^-+c_0^+)$, we find that
$$ 
\Vert e^{-s\eta(\cdot,0)} A \Vert^2_{L^2(\Omega)^n} \leq C_2' 
\left( \Vert e^{-s\eta(\cdot,0)} p \Vert^2_{L^2(\Omega)} + \Vert e^{-s\eta(\cdot,0)} q^+ \Vert^2_{L^2(\Omega)} + \Vert e^{-s\eta(\cdot,0)} q^- \Vert^2_{L^2(\Omega)} + s^{-1/2} \sum_{k=1}^n \mu_{k+2}(s) \right).
$$
From this and \eqref{p6} it then follows for all $s>s_*=\max(s_2,s_2',(2 C_2 C_2')^{2 \slash 3})$ that 
$$
\Vert e^{-s\eta(\cdot,0)}A\Vert^2_{L^2(\Omega)^n}+\Vert e^{-s\eta(\cdot,0)}p\Vert^2_{L^2(\Omega)}+\Vert e^{-s\eta(\cdot,0)}q^+\Vert^2_{L^2(\Omega)}+\Vert e^{-s\eta(\cdot,0)}q^-\Vert^2_{L^2(\Omega)}
\leq C_3 \sum_{k=1}^{n+2} \mu_k(s),
$$
with $C_3=(1+C_2) C_2'$. Thus, bearing in mind that $0 \leq \eta(x,0) \leq e^{2\lambda K} \slash T^2$ for all $x \in \Omega$,
we find upon setting $C_4=e^{(2 s_* e^{2\lambda K}) \slash T^2} C_3$ that
$$
\Vert A\Vert^2_{L^2(\Omega)^n}+\Vert p\Vert^2_{L^2(\Omega)}+\Vert q^+\Vert^2_{L^2(\Omega)}+\Vert q^-\Vert^2_{L^2(\Omega)}
\leq C_4 \sum_{k=1}^{n+2} \mu_k(s_*).$$
Now the result follows from this upon noticing that 
$$ \mu^{\pm,k}(s_*)= 2 \norm{ e^{-s_*\eta(\cdot,0)}\varphi^{1/2}(\partial_\nu\beta)^{1/2} \partial_\nu v^{\pm,k}}_{L^2(\Sigma_*)}^2 \leq 
C \norm{\partial_\nu v^{\pm,k}}_{L^2(\Sigma_*)}^2,\ k=1,\ldots,n+2, $$
with $C=2 \norm{e^{-\eta(\cdot,0)}\varphi^{1/2}(\partial_\nu\beta)^{1/2}}_{L^\infty(\Sigma_*)}^2 < +\infty$, since
$\varphi(x,t)=\varphi(x,-t)$, $\eta(x,t)=\eta(x,-t)$ and $v^{\pm,k+2}(x,t)=-\overline{v^{\pm,k+2}(x,-t)}$ for all $(x,t)\in \Omega \times (-T,0)$.

\section{Appendix}
\subsection{Appendix A: Symmetry}
\label{sec-aA}
To prove that $\tilde{A} \cdot\nabla$ is symmetric it is enough to see that
for all $u=(u^+,u^-)^T$ and $v=(v^+,v^-)^T$ in $(H_0^1(\Omega))^2$, that we have:
\begin{eqnarray*}
\langle \tilde{A} \cdot\nabla u , v \rangle_{L^2(\Omega)^2} & = & \int_{\Omega} A \cdot \nabla u^- \overline{v^+} dx - 
\int_{\Omega} A \cdot \nabla u^+ \overline{v^-} dx \\
& = & \int_{\Omega} (\nabla \cdot (A u^-)) \overline{v^+} dx - 
\int_{\Omega} (\nabla \cdot (A \cdot u^+) ) \overline{v^-} dx \\ 
& = & -\int_{\Omega} u^- \cdot A \overline{v^+} dx +
\int_{\Omega} u^+ \cdot A \overline{v^-} dx \\ 
& = & \int_{\Omega} u^+ \cdot \overline{A v^-} dx - \int_{\Omega} u^- \cdot \overline{A v^+} dx \\ 
& = & \langle u , \tilde{A} \cdot \nabla v \rangle_{L^2(\Omega)^2}.
\end{eqnarray*}
Notice that we used the facts that $A$ is real-valued and divergence free, i.e. that $\nabla \cdot A=0$.

\subsection{Appendix B: $\Delta^D$-bounded perturbation}
\label{sec-aB}
Each of the three operators $\tilde{p}$, $\tilde{q}$ and $\tilde{A} \cdot \nabla$ is $\Delta^D$-bounded with relative bound zero. Indeed, for all $u=(u^+,u^-)^T\in \mathcal{D}_0^2$, we have
$$ \norm{\tilde{p} u}_{L^2(\Omega)^2}^2 = \norm{p u^-}_{L^2(\Omega)}^2 + \norm{p u^+}_{L^2(\Omega)}^2
\leq 2 \norm{p}_{L^{\infty}(\Omega)}^2 \norm{u}_{L^2(\Omega)^2}^2, $$
$$ \norm{\tilde{q} u}_{L^2(\Omega)^2}^2 = \norm{q^+ u^+}_{L^2(\Omega)}^2 + \norm{q^- u^-}_{L^2(\Omega)}^2
\leq \left( \norm{q^+}_{L^{\infty}(\Omega)}^2 + \norm{q^-}_{L^{\infty}(\Omega)}^2 \right) \norm{u}_{L^2(\Omega)^2}^2 $$
and
\begin{eqnarray*}
\norm{\tilde{A} \cdot \nabla u}_{L^2(\Omega)^2}^2 
& = & \norm{A \cdot u^-}_{L^2(\Omega)}^2 + \norm{A \cdot u^+}_{L^2(\Omega)}^2\\
& \leq & \norm{A}_{L^{\infty}(\Omega)^n}^2 \left( \norm{\nabla u^+}_{L^2(\Omega)}^2+  \norm{u^-}_{L^2(\Omega)}^2 \right)\\
& \leq & \norm{A}_{L^{\infty}(\Omega)^n}^2 \int_\Omega (-\Delta u^+ \overline{u^+} - \Delta u^- \overline{u^-}) dx\\
& \leq & \norm{A}_{L^{\infty}(\Omega)^n}^2 \left( \norm{\Delta u^+}_{L^2(\Omega)} \norm{u^+}_{L^2(\Omega)} + \norm{\Delta u^-}_{L^2(\Omega)} \norm{u^-}_{L^2(\Omega)} \right)
\end{eqnarray*}
by the Cauchy-Schwarz inequality, with 
\begin{eqnarray*}
& & \norm{\Delta u^+}_{L^2(\Omega)} \norm{u^+}_{L^2(\Omega)} + \norm{\Delta u^-}_{L^2(\Omega)} \norm{u^-}_{L^2(\Omega)} \\
& \leq & \norm{A}_{L^{\infty}(\Omega)^n}^2 \left( \frac{\epsilon}{\norm{A}_{L^{\infty}(\Omega)^n}^2} \left( \norm{\Delta u^+}_{L^2(\Omega)}^2 +  \norm{\Delta u^-}_{L^2(\Omega)}^2 \right) +  \frac{\norm{A}_{L^{\infty}(\Omega)^n}^2}{\epsilon} \left( \norm{u^+}_{L^2(\Omega)}^2  + \norm{u^-}_{L^2(\Omega)}^2 \right) \right)
\end{eqnarray*}
uniformly in $\epsilon \in (0,1)$, upon applying the H\"older inequality, whence
\begin{eqnarray*}
\norm{\tilde{A} \cdot \nabla u}_{L^2(\Omega)^2}^2
& \leq & \epsilon \norm{\Delta^D u}_{L^2(\Omega)^2}^2 + \frac{\norm{A}_{L^{\infty}(\Omega)^n}^4}{\epsilon} \norm{u}_{L^2(\Omega)^2}^2.
\end{eqnarray*}

\section*{Acknowledgments}
A. K. is supported by the Austrian Science Fund (FWF) under Project No. M 2310-N32.
\'E. S. is partially supported by the Agence Nationale de la Recherche (ANR) under grant ANR-17-CE40-0029.


\bigskip \bigskip

\noindent {\bf Andrii Khrabustovskyi} \\
Technische Universit\"at Graz, Institut f\"ur Angewandte Mathematik, Steyrergasse 30 , 8010 Graz, Austria. \\
E-mail: khrabustovskyi@math.tugraz.at.
\bigskip

\noindent {\bf Imen Rassas} \\
Universit\'e de Tunis El Manar, \'Ecole Nationale d'Ing\'enieurs de Tunis, LAMSIN, BP 37, Tunis Le Belv\'ed\`ere, Tunisia.\\
E-mail: imen.rassass@gmail.com. 
\bigskip

\noindent {\bf \'Eric Soccorsi} \\
Aix-Marseille Univ, Universit\'e de Toulon, CNRS, CPT, Marseille, France. \\
E-mail: eric.soccorsi@univ-amu.fr.

\end{document}